\documentclass[12pt]{amsart}

\usepackage{amsfonts,amssymb,stmaryrd,amscd,amsmath,latexsym,amsbsy}

\usepackage{amssymb}
\usepackage{amsfonts}
\usepackage{latexsym}
\usepackage{bbold}

\newtheorem{theorem}{Theorem}[section]
\newtheorem{lemma}[theorem]{Lemma}
\newtheorem{proposition}[theorem]{Proposition}
\newtheorem{corollary}[theorem]{Corollary}
\theoremstyle{definition}
\newtheorem{definition}[theorem]{Definition}

\newtheorem{example}[theorem]{Example}

\newtheorem{remark}[theorem]{Remark}

\newcommand{\Hh}{\mathcal{H}}
\newcommand{\Tr}{\text{Tr}}

\newcommand{\End}{\text{End}}
\newcommand{\Hom}{\text{Hom}}

\newcommand{\Rep}{\text{Rep}}

\newcommand{\h}{\mathfrak{h}}

\newcommand{\ben}{\begin{enumerate}}
\newcommand{\een}{\end{enumerate}}

\newcommand{\mC}{{\mathcal C}}

\theoremstyle{plain}

\newtheorem*{sol}{Solution}

\theoremstyle{definition}

\theoremstyle{remark}

\newcommand{\solu}[1]{\begin{sol}{\bf (\ref{#1})}}

\begin{document}

\title{Representation theory in complex rank, I}

\author{Pavel Etingof}
\address{Department of Mathematics, Massachusetts Institute of Technology,
Cambridge, MA 02139, USA}
\email{etingof@math.mit.edu}
\maketitle

\vskip .05in
\centerline{\bf Dedicated to E. B. Dynkin on his 90th birthday}
\vskip .05in

\section{Introduction}

The subject of representation theory in complex rank goes back to the papers \cite{DM,De1}.  
Namely, these papers introduce Karoubian tensor categories $\Rep(GL_t)$ (\cite{DM,De1}), $\Rep(O_t)$,
$\Rep(Sp_{2t})$, $t\in \Bbb C$ (\cite{De1}), which are interpolations of the tensor categories of
algebraic representations of classical complex algebraic groups 
$GL_n$, $O_n,Sp_{2n}$ to non-integral rank\footnote{In fact, $\Rep (O_t)=\Rep (Sp_{-t})$ with a modified symmetric structure.}. 
This means that when $t=n$ is a nonnegative integer, these categories project onto the corresponding classical representation categories
${\bold{Rep}}(GL_n)$, ${\bold{Rep}}(O_n)$, ${\bold{Rep}}(Sp_{2n})$, i.e., 
they have tensor ideals, the quotients by which are the classical representation categories.  
Later, in \cite{De2}, P. Deligne introduced Karoubian tensor categories $\Rep(S_t)$, $t\in \Bbb C$, which are similar interpolations 
for the representation category of the symmetric group ${\bold{Rep}}(S_n)$ (and project onto it for $t=n$). 

In \cite{Kn1,Kn2}, F. Knop proposed a broad generalization of
Deligne's construction. In particular, he generalized 
his construction for $S_n$ to the case of wreath product groups $S_n\ltimes
\Gamma^n$, where $\Gamma$ is any finite group, constructing 
Karoubian tensor categories $\Rep(S_t\ltimes \Gamma^t)$ for complex
$t$, projecting for $t=n$ onto ${\bold{Rep}}(S_n\ltimes \Gamma^n)$. 

Since these categories are semisimple for non-integer $t$, 
one may think of these results as ''compact" representation 
theory in complex rank. The goal of this paper is to start 
developing the ''noncompact" counterpart of this theory.
Namely, in this paper we will introduce a method that allows one to 
define interpolations to complex rank of various categories of representations of classical type,
in particular the following ones:

1) wreath products; 

2) degenerate affine Hecke algebras; 

3) rational and trigonometric Cherednik algebras, 
symplectic reflection algebras; 

4) real groups (i.e., symmetric pairs);

5) Lie superalgebras; 

6) affine Lie algebras; 

7) Yangians;

8) (Parabolic) category O 
for reductive Lie algebras.

Namely, we will define representations 
of a ''noncompact algebra" of complex rank as representations of its 
''maximal compact subalgebra" (i.e. an (ind)-object of the corresponding tensor category) 
together with some additional structure (morphisms satisfying some relations). 
These morphisms and relations are obtained by writing down a ``categorically friendly'' 
definition of the corresponding classical structure, and then interpolating this definition to 
complex values of the rank parameter(s). We will discuss the explicit form of such morphisms and 
relations in the above special cases 1-3, based on $\Rep(S_t)$. Cases 4-8 based on $\Rep(GL_t)$, $\Rep(O_t)$, 
$\Rep(Sp_{2t})$ and many other similar situations will be considered in future papers. 
\footnote{The ideas explored in this paper were outlined in the talk \cite{E} in March 2009.}

This approach leads to a multitude of new interesting representation categories, 
which, in a sense, capture the phenomenon of ''stabilization with respect to rank" 
in representation theory of classical groups and algebras.
In subsequent works, we plan to study the structure of these categories in detail. 
We expect to discover interesting degeneration phenomena not only at integer $t$, but also at rational 
(non-integer) values of $t$. Specifically, if the ''classical" theory already has a continuous parameter $k$ (or several such parameters), 
then we expect interesting degeneration phenomena when both $k$ and $t$ are rational but not necessarily integer. 
Such phenomena cannot be understood by interpolation arguments, and their study will likely require new ideas. 
For instance, for rational Cherednik algebras, such degeneration phenomena were discovered 
in \cite{EA1}. Also, one might hope that these new categories will be interesting for the theory of categorification. 
 
For simplicity, we consider only the case $t\notin \Bbb Z_+$, when the Deligne categories $\Rep(S_t)$ are semisimple, although 
many of the results admit generalizations to the case $t\in \Bbb Z_+$. 

The organization of the paper is as follows. 

In Section 2 we recall basic facts about the Deligne category $\Rep(S_t)$. In particular, we 
define the group algebra $\Bbb C[S_t]$ of $S_t$ and discuss its center, and its action on simple objects of $\Rep(S_t)$. 

In Section 3, we discuss the Schur-Weyl dualty for $\Rep(S_t)$,
using the notion of a complex tensor power of a vector space 
with a distinguished nonzero vector. 

 In Section 4, we apply the Schur-Weyl duality from Section 3 to 
define and study wreath products of complex rank. 

Finally, in Section 5 we discuss interpolations of degenerate affine Hecke algebras 
and of wreath product symplectic reflection algebras, in particular rational 
Cherednik algebras. 

\begin{remark}
The world of representation categories of complex rank is vast and almost unexplored, 
and there are a multitude of projects that immediately spring to mind, to interpolate 
various settings in classical representation theory. Many of them seem interesting, and it is not yet clear which 
ones are most worthy of exploration. 

It seems that a good strategy of choosing problems 
in this field at initial stages would be to keep closer to applications to ''classical" representation theory.
One of the pathways for such applications is to consider the algebra $A=\End(X)$, where $X$ is an 
(ind-)object of some complex rank category $\mathcal{C}$. This is an ordinary algebra that belongs to the world of ''classical" representation theory, 
but its construction as $\End(X)$ is sometimes more insightful than ''classical" constructions, if they exist at all. 
In particular, through such a construction, $A$ comes equipped with a large family
of modules, namely $\Hom(X,Y)$, where $Y$ is another object of ${\mathcal C}$. 
  
Another principle that might be useful is to focus on extracting concrete numerical information (such as characters, multiplicities, etc.), as this has been a fundamental principle of representation theory since its creation. To this end, it may be useful to relate complex rank categories 
to ''classical" representation categories through various functorial constructions, such as Schur-Weyl duality and its generalizations.    
\end{remark}

{\bf Acknowledgments.} The author is grateful to I. Entova-Aizenbud and V. Ostrik for many useful discussions. 
The work of the author was  partially supported by the NSF grants
DMS-0504847 and DMS-1000113.

\section{Deligne categories $\Rep(S_t)$}

\subsection{Categorical preliminaries}

All categories we consider will be additive and linear over $\Bbb C$.  

Recall that a Karoubian category is an additive category closed under taking direct summands.  

By a tensor category, unless otherwise specified, we will mean a Karoubian rigid monoidal category 
with additive tensor product. 

For a category ${\mathcal {C}}$ denote by ${\rm Ind}{\mathcal{C}}$ the ind-completion of ${\mathcal{C}}$. 
If ${\mathcal C}$ is semisimple, objects of ${\rm Ind}({\mathcal{C}})$ are possibly infinite 
direct sums of indecomposable objects of ${\mathcal C}$. 
By an ind-object in ${\mathcal{C}}$ we will mean an object of ${\rm Ind}{\mathcal{C}}$. 

If ${\mathcal C}$, ${\mathcal D}$ are semisimple categories, by ${\mathcal C}\boxtimes {\mathcal D}$ 
we denote their external tensor product, -- a semisimple category whose simple objects have the form
$X\otimes Y$, where $X$ runs over simple objects of ${\mathcal C}$
and $Y$ over simple objects of ${\mathcal D}$.  It is clear that if ${\mathcal C}, {\mathcal D}$ are tensor categories 
then the category ${\mathcal C}\boxtimes {\mathcal D}$ is also a tensor category. 

\subsection{Definition of the Deligne category ${\rm Rep}(S_t)$ and its basic properties}

Let us recall basic facts about the 
Deligne category $\Rep(S_t)$ (\cite{De1}). 
\footnote{A good source of materials on Deligne categories is the webpage of the MIT seminar on Deligne categories, http://math.mit.edu/\~{}innaento/DeligneCatSeminar/}
This is a Karoubian rigid tensor category 
over $\Bbb C$ defined for any complex number $t$. 
Indecomposable objects $X_\lambda$ of the category $\Rep(S_t)$ are labeled by 
all Young diagrams $\lambda$.  

If $t$ is not a nonnegative integer, then 
$\Rep(S_t)$ is semisimple abelian, and its 
simple objects $X_\lambda$ 
are labeled by all Young diagrams (or partitions) $\lambda$. 
In particular, the empty diagram corresponds to the 
neutral object $\bold 1$, the one-box diagram to the ``reflection
representation'' $\h_0$, the column of length $k$ to $\wedge^k\h_0$,
etc. The object $X_\lambda$ interpolates the representations of
$S_n$ corresponding to the partition
$$
\widetilde{\lambda}(n):=(n-|\lambda|,\lambda_1,...,\lambda_m,...)
$$ 
of $n$, for large $n$. 

On the other hand, if $t$ is a nonnegative integer $n$, then $\Rep(S_t)$ projects onto the 
category of representations of the symmetric group ${\bold{Rep}}(S_n)$,
whose simple objects are the irreducible representations $\pi_\mu$ 
of $S_n$ attached to Young diagrams (or partitions) $\mu$ with $|\mu|=n$.
Namely, $\Rep(S_t)$ contains a tensor ideal ${\mathcal I}_n$, such that 
$\Rep(S_t)/{\mathcal I}_n={\bold{Rep}}(S_n)$. The partitions $\lambda$ for which 
$X_\lambda$ have a nonzero image in the quotient are those 
for which $\lambda_1+|\lambda|\le n$, and for such $\lambda$ 
the object $X_\lambda$ maps to $\pi_{\widetilde{\lambda}(n)}\in {\bold{Rep}}(S_n)$. 

Deligne also defined 
the abelian category $\Rep^{ab}(S_t)$, which is ``the abelian envelope''
of $\Rep(S_t)$ (it differs from $\Rep(S_t)$ only for nonnegative 
integer $t$). 

The categorical dimensions of indecomposable 
objects in the category $\Rep(S_t)$ for $t\notin \Bbb Z_+$ are given by the
following formula. For each partition $\lambda$ with
$|\lambda|=N$, let $\lambda^*$ be the conjugate partition, and 
define the set $B_\lambda$ to be the set of all nonnegative 
integers not contained in the sequence $N-1+k-\lambda_k^*$, 
$k\ge 1$ (since this sequence is increasing, it is clear 
that $|B_\lambda|=N$). Then we have 

\begin{proposition}\label{dimfor} (see \cite{De2}, \cite{CO}) If $t\notin \Bbb Z_+$, then 
$$
\dim X_\lambda=\dim \pi_\lambda 
\frac{\prod_{k\in B_\lambda}(t-k)}{N!}=\frac{\prod_{k\in B_\lambda}(t-k)}{\prod_{(i,j)\in \lambda}h_\lambda(i,j)}, 
$$ 
where $h_\lambda(i,j)$ are the hooklengths in $\lambda$.
\end{proposition}

This formula is obtained 
by interpolating the hooklength formula for the dimensions of
$\pi_{\widetilde{\lambda}(n)}$.

\begin{example}\label{dimforex}
1. Let $\lambda=(1^k)$ ($k$ times).
Then $\dim X_\lambda=\frac{(t-1)...(t-k)}{k!}=\binom{t-1}{k}$. 

2. Let $\lambda=(k)$. Then $\dim X_\lambda=\frac{(t-2k+1)\prod_{j=0}^{k-2}(t-j)}{k!}=\binom{t}{k}-\binom{t}{k-1}$.
\end{example} 

\begin{remark} Formula (\ref{dimfor}) is not always true for nonnegative integer $t$.
For instance, as we see from Example \ref{dimforex}(2), the polynomial 
giving $\dim X_{(k)}$ for generic $t$ is negative for $t=k-1,...,2k-2$. 
This means that $\dim X_{(k)}$ cannot be given by 
this formula for these values of $k$ (as then it will have nothing to be mapped to in the category 
of representations of the symmetric group, as the functor from $\Rep(S_t)$ to this category is a symmetric tensor functor and hence 
preserves dimensions). In fact, one can show (see \cite{CO}) that $X_\lambda$ for positive integers $t$ 
is not always a specialization of a simple object for generic $t$; its lift to generic $t$ may be reducible.   
\end{remark}

It is also useful to recall the rule of multiplication by $\h_0$. 

\begin{proposition} (see \cite{De2},\cite{CO})
One has 
$$
\h_0\otimes X_\lambda=\oplus_{\mu\in P_\lambda^+\cup
P_\lambda^-\cup P_\lambda^0}X_\mu+{\rm cc}(\lambda)X_\lambda,  
$$
where $P_\lambda^+, P_\lambda^-, P_\lambda^0$ are the sets 
of Young diagrams obtained from $\lambda$ by adding, deleting,
and moving a corner cell, respectively, and ${\rm cc}(\lambda)$ 
is the number of corner cells of $\lambda$. 
\end{proposition}

This formula is obtained by interpolating the Pieri rule. 

\subsection{The universal property, induction, and restriction}

For $t\notin \Bbb Z_+$, the category $\Rep(S_t)$ is 
known (see \cite{De1}, Section 8) 
to be the universal symmetric tensor category with a commutative 
Frobenius algebra $\h$ of dimension $t$. This means that 
for any symmetric tensor category $\mC$, tensor functors 
$\Rep(S_t)\to \mC$ correspond of commutative Frobenius 
algebras in $\mC$ of dimension $t$. 

Now let $t_1,...,t_m\in \Bbb C$ be such that $\sum_{i=1}^mt_i=t$. 
Consider the category $\boxtimes_{i=1}^m \Rep(S_{t_i})$. 
In this category we have the commutative Frobenius algebra 
$\h_1\oplus...\oplus \h_m$, where 
$$
\h_i:=\bold 1^{\boxtimes i-1}\boxtimes 
\h_{\Rep(S_{t_i})}\boxtimes \bold 1^{\boxtimes m-i}.
$$
The dimension of this algebra is $t$. 
So by virtue of the universal property, we 
have a symmetric tensor functor 
$$
{\rm Res}_{t_1,...,t_n}: \Rep(S_t)\to \boxtimes_{i=1}^m\Rep(S_{t_i}),
$$
such that ${\rm Res}_{t_1,...,t_n}(\h)=\h_1\oplus...\oplus \h_m$.
The left (respectively, right) adjoint of this functor 
is called the induction (respectively, coinduction)  functor, and denoted by ${\rm (Co)Ind}_{t_1,...,t_m}$
(it lands in the ind-completion of $\Rep(S_t)$). 
In the special case $m=2$ and $t_2\in \Bbb Z_+$ (with ${\bold{Rep}}(S_{t_2})$ instead of $\Rep(S_{t_2})$), 
these functors are considered 
in \cite{De1}.

\subsection{The Jucys-Murphy central element}

Recall that the Jucys-Murphy central element 
is the central element $\Omega=\sum_{1\le i<j\le n}s_{ij}$ in $\Bbb C[S_n]$, where 
$s_{ij}$ is the transposition of $i$ and $j$. It is well
known that $\Omega$ acts on an irreducible representation $\pi_\mu$
by the scalar ${\rm ct}(\mu)$ (the content of $\mu$),
which is the sum of $i-j$ over all cells $(i,j)$ of 
the Young diagram $\mu$. 

The interpolation of the Jucys-Murphy central element to complex values of
$t$ (as an endomorphism of the identity functor of the category
$\Rep(S_t)$) is constructed in the paper \cite{CO} (as a special case of one-cycle central elements, see below). 
When $t\notin \Bbb Z_+$, t is easy to describe this endomorphism explicitly. Indeed,  
it is easy to see that 
$$
{\rm ct}(\widetilde{\lambda}(n))={\rm ct}(\lambda)-|\lambda|+
\frac{(n-|\lambda|)(n-|\lambda|-1)}{2}.
$$
So we define the endomorphism $\Omega$ by the formula  
$$
\Omega|_{X_\lambda}:={\rm ct}(\lambda)-|\lambda|+
\frac{(t-|\lambda|)(t-|\lambda|-1)}{2}.
$$

\subsection{Higher central elements}

In a similar way one can interpolate other central elements 
of the group algebra $\Bbb C[S_n]$, corresponding to various cycle types. 
Namely, let ${\bold m}=(m_1,m_2,...)$ be a sequence of nonnegative integers with $\sum m_i(i+1)\le n$. 
Let $m=\sum m_i(i+1)$. Then in the group algebra 
$\Bbb C[S_n]$ we have the central element $\Omega_{\bold m}$, 
which is the sum of all permutations with $m_i$ cycles of length $i+1$ for each $i\ge 1$. 
The eigenvalue of $\Omega_{\bold m}$ on $\pi_\lambda$ 
equals 
$$
\Omega_{\bold m}|_{\pi_\lambda}= \frac{|C_{\bold m}|\cdot {\rm tr}|_{\pi_\lambda}(g)}{\dim \pi_\lambda},
$$
where $C_{\bold m}$ is the conjugacy class of permutations with $m_i$ cycles of length $i+1$ for each $i\ge 1$, and 
$g\in C_{\bold m}$. We have 
$$
|C_{\bold m}|=\frac{n(n-1)...(n-m+1)}{\prod_i m_i! (i+1)^{m_i}}.
$$
This implies that the interpolation of $\Omega_{\bold m}$ to $\Rep(S_t)$ is given by the formula
$$
\Omega_{\bold m}|_{X_\lambda}=\frac{\prod_{(i,j)\in \lambda}h_\lambda(i,j)\cdot \prod_{j=0}^{m-1}(t-j)}{\prod_i m_i! (i+1)^{m_i} \cdot 
\prod_{k\in B_\lambda}(t-k)}c_{\lambda,\bold m}(t),
$$
where $c_{\lambda,\bold m}(t)$ is the coefficient of $x^\lambda:=\prod_i x_i^{\lambda_i}$ 
in the series 
$$
F_{t,{\bold m}}(x):=(1+p_1(x))^{t-m}\prod_{i\ge 1}
(1+p_{i+1}(x))^{m_i}\prod_{i\ge 1} (1-x_i)\prod_{i>j}(1-\frac{x_i}{x_j}),
$$
where $p_i(x)=\sum_r x_r^i$
(it is clear that $c_{\lambda,{\bold m}}(t)$ is a polynomial, 
by setting $x_i=u_1...u_i$ and writing the expansion in terms of
the $u_i$). This interpolation is obtained from the Frobenius character formula and the hooklength formula 
for $S_n$-representations, and for the case of one cycle it is given in \cite{CO}. In particular, $\Omega_{1,0,0...}=\Omega$ 
(the Jucys-Murphy element).\footnote{Note that the element $\Omega_{\bold m}$ is a polynomial 
of the elements $Z_1,Z_2,...$, where $Z_i:=\Omega_{\bold e_i}$ are the one-cycle central elements discussed in \cite{CO} 
($(\bold e_i)_j=\delta_{ij}$). For this reason, the one-cycle elements suffice for the study of blocks of $\Rep(S_t)$, done in \cite{CO}; 
more general central elements don't carry additional information.} 

Note that $\Omega_{\bold m}|_{X_\lambda}$, as well as $\dim X_\lambda$, is an integer-valued polynomial of $t$ (i.e., an integer linear combination of 
binomial coefficients). Indeed, this function takes integer values at large positive 
integers $t$ (by representation theory of symmetric groups), and such a rational function is well known to be an integer-valued polynomial. 

\subsection{The group algebra $\Bbb C[S_t]$ of $S_t$}

The representation category of the symmetric group $S_n$ may be defined as the tensor category of representations 
of the Hopf algebra $\Bbb C[S_n]$. In the setting of Deligne categories, such an algebra 
can also be defined, as a Hopf algebra in $\Rep(S_t)$ (which interpolates the group algebra of $S_n$ with the conjugation action of $S_n$). 
The only caveat is that since $|S_n|=n!$ is not a polynomial in $n$, this algebra will be infinite dimensional (i.e., an ind-object of $\Rep(S_t)$). 

Namely, let us fix a cycle type $\bold m$ as above. Then we have the conjugacy class $C_{\bold m}$ in $S_n$, 
and the span of this conjugacy class in $\Bbb C[S_n]$ is a representation of $S_n$ (with the conjugation action). 
This representation is the induced representation 
$$
{\rm Ind}_{S_{n-m}\times \prod S_{m_i}\ltimes (\Bbb Z/(i+1)\Bbb
  Z)^{m_i}}^{S_n}\Bbb C.
$$ 
(where $S_{m_i}\ltimes (\Bbb Z/(i+1)$ is viewed as a subgroup of $S_n$) 
This induced representation has a natural analog in $\Rep(S_t)$, namely the invariants $E_{\bold m}$ 
of the group $\prod_i S_{m_i}\ltimes (\Bbb Z/(i+1)\Bbb Z)^{m_i}$ in the object $\Delta_{m,t}={\rm Ind}_{S_{t-m}}^{S_t}(\bold 1)$
(where ${\rm Ind}_{S_{t-m}}^{S_t}$ is the induction functor $\Rep(S_{t-m})\to \Rep(S_t)$). 
We also have a natural multiplication map 
$$
{\rm mult}_{\bold m,\bold m',\bold m''}: E_{\bold m}\otimes E_{\bold m'}\to E_{\bold m''},
$$ 
which interpolates the multiplication in $\Bbb C[S_n]$
(i.e., in the classical setting for $s\in C_{\bold m}$, $s'\in C_{\bold m'}$ one has ${\rm mult}_{\bold m,\bold m',\bold m''}(s,s')=ss'$ 
if $ss' \in C_{\bold m''}$, and zero otherwise, and we need to interpret this in terms of the diagrams of the partition algebra). 
Moreover, it is clear that for fixed $\bold m$ and $\bold m'$ this map 
is zero for almost all $\bold m''$. This shows that $\Bbb
C[S_t]:=\oplus_{\bold m}E_{\bold m}$ is an associative algebra in
${\rm Ind}\Rep(S_t)$. Note that this algebra is $\Bbb Z/2\Bbb
Z$-graded by ``parity of the permutation'', $\deg(E_{\bold m})=\sum_i im_i \text{ mod }2$. 

Moreover, $\Bbb C[S_t]$ is a cocommutative Hopf algebra, in
which the coproduct $\Delta: E_{\bold m}\to E_{\bold m}\otimes
E_{\bold m}$ and counit $\varepsilon: E_{\bold m}\to \bold 1$ come from the natural commutative algebra structure 
on $E_{\bold m}^*$ (note that $E_{\bold m}^*$ interpolates the space of functions on $C_{\bold m}$). It is also easy to construct the antipode 
$S: E_{\bold m}\to E_{\bold m}$, interpolating the inversion map 
on $S_n$. Namely, this is the antiautomorphism 
which is the identity on $E_{1,0,...}$. Thus, the category $\Bbb
C[S_t]$-mod of $\Bbb C[S_t]$-modules in ${\rm Ind}\Rep(S_t)$ is a
symmetric tensor category. 

Note that we have a canonical tensor functor 
$\Rep(S_t)\to \Bbb C[S_t]$-mod; in particular, any 
object of $\Rep(S_t)$ carries a canonical action 
of $\Bbb C[S_t]$. Indeed, to prescribe such a functor, 
it suffices to specify a Frobenius algebra in $\Bbb C[S_t]$-mod 
of dimension $t$. For this, it suffices to define a Hopf 
action of $\Bbb C[S_t]$ on the Frobenius algebra $\h$, which is done in
a straightforward way by interpolating from integer $n$ 
the standard action of $\Bbb C[S_n]$ on 
${\rm Fun}(\lbrace{1,...,n\rbrace})$. 

\begin{remark}\label{semidir} 
Let $B$ be an algebra in ${\rm IndRep}(S_t)$. Then, since $B$ is an object of 
${\rm IndRep}(S_t)$, there is a standard action of $\Bbb C[S_t]$ 
on $B$. It is easy to check that this is a Hopf algebra 
action of the Hopf algebra $H:=\Bbb C[S_t]$ on the algebra $B$  inside ${\rm IndRep}(S_t)$,
$\rho: H\otimes B\to B$. 
Thus, we can form the semidirect product $\Bbb C[S_t]\ltimes B$, as we do in the
theory of Hopf algebra actions on rings. Namely, as an object, this 
semidirect product is $B\otimes H$, with multiplication map 
$m: B\otimes H\otimes B\otimes H\to B\otimes H$ is defined 
by the formula
$$
m=(m_H\otimes m_B)\circ (1\otimes \rho\otimes 1\otimes 1)\circ \sigma_{34}\circ (1\otimes \Delta_H\otimes 1\otimes 1),
$$
where $m_B,m_H$ are the products in $B,H$, $\Delta_H$ is the coproduct in $H$, 
and $\sigma_{34}$ is the permutation of the 3d and 4th components. 
Now, the category of $B$-modules in ${\rm IndRep}(S_t)$ is tautologically equivalent to 
the category of $\Bbb C[S_t]\ltimes B$-modules, in which 
$\Bbb C[S_t]$ acts via the standard action.
\footnote{Note that the latter requirement is 
needed, as there are many ${\Bbb C}[S_t]$-modules 
in which the action of $\Bbb C[S_t]$ is different from the standard 
action of $\Bbb C[S_t]$ 
on the underlying object of ${\rm IndRep}(S_t)$, see Subsection \ref{hcmod} below.} 
Moreover, if $B$ is a Hopf algebra then so is $\Bbb C[S_t]\ltimes B$, and the above equivalence 
of categories is a tensor equivalence. 
\end{remark}

\begin{remark}
It is easy to see that for each $\bold m$, $E_{\bold m}$ has 
a unique invariant up to scaling, which we denote by $\Omega_{\bold m}$. 
This notation is justified by the fact that the 
aforementioned canonical action of this invariant
on $X_\lambda$ coincides with the endomorphism 
$\Omega_{\bold m}$ defined in the previous subsection. 
Thus, ${\rm Hom}(\bold 1,\Bbb C[S_t])={\rm
  Span}(\lbrace{\Omega_{\bold m}\rbrace})$
is a commutative algebra with basis 
$\Omega_{\bold m}$. It is easy to show that 
this algebra is the polynomial algebra 
$\Bbb C[Z_1,Z_2,...]$ on the generators $Z_j=\Omega_{j,0,0,...}$, 
a fact exploited in \cite{CO}.  
\end{remark}

\subsection{A presentation of $\Bbb C[S_t]$ by generators and relations}
It is well known that the group algebra 
$\Bbb C[S_n]$ with generators 
being simple reflections $s_{ij}$
is a quadratic algebra. Namely, the defining relations are: 
\begin{equation}\label{snrel}
s_{ij}^2=1,\ s_{ij}s_{jk}=s_{ik}s_{ij},\
s_{ij}s_{kl}=s_{kl}s_{ij}, 
\end{equation}
where different subscripts denote different indices. 
It is easy to interpolate this
presentation to the case of complex rank, which would 
yield an inhomogeneous quadratic presentation 
of the algebra $\Bbb C[S_t]$, generated by
$E=E_{1,0,0,...}=\Delta_{2,t}^{S_2}$, valid at least for transcendental $t$.  

One may consider the filtration on $\Bbb C[S_t]$ 
defined by the condition that $\deg(E)=1$, and the 
associated graded algebra ${\rm gr}(\Bbb C[S_t])$. 
Similarly to the case of $S_n$ for integer $n$, this 
algebra is generated by the same generators with 
defining relations being the homogenization of the 
relations of $\Bbb C[S_t]$. 
Indeed, for $S_n$ the homogenized relations look like
$$
s_{ij}^2=0,\ s_{ij}s_{jk}=s_{ik}s_{ij},\
s_{ij}s_{kl}=s_{kl}s_{ij}, 
$$
and modulo these relations, any nonzero monomial 
in $s_{ij}$ may be rewritten as $s_{i_1j_1}...s_{i_mj_m}$, so
that the function ${\rm max}(i_l,j_l)$ strictly 
increases from left to right, and there are exactly $n!$ such
ordered monomials. This also shows that the ordered 
monomials are linearly independent, and that 
the quadratic algebra ${\rm gr}(\Bbb C[S_n])$ is Koszul
\footnote{I am grateful to Eric Rains for this remark.}
(the quadratic Gr\"obner basis is formed 
by the unordered quadratic monomials, i.e. 
$s_{ij}s_{pq}$ with ${\rm max}(i,j)\ge {\rm max}(p,q)$). 
By interpolation, the same statements are true for $\Bbb C[S_t]$,
at least for transcendental $t$. 

Also, one can show that the Hilbert series of the algebra ${\rm gr}(\Bbb C[S_t])$ 
is 
$$
h(t,x)=x^t\frac{\Gamma(x^{-1}+t)}{\Gamma(x^{-1})},
$$
where the function on the right should be replaced by its
asymptotic expansion at $x\to +\infty$. Note that 
formal substitution of $x=1$ on the RHS (``order of $S_t$'') gives
$h_t(1)=\Gamma(1+t)$, which is $t!$ for integer $t$, but this is
illegitimate, as the series on the right side has zero radius of
convergence, even though the algebra ${\rm gr}(\Bbb C[S_t])$ 
is ''finitely generated" (i.e. generated by an honest object, 
not just an ind-object, of $\Rep(S_t)$, namely the object $E$).
\footnote{Note that such a thing clearly cannot happen for ordinary finitely
generated graded algebras. Indeed, if $A$ is a graded algebra with 
generators $x_1,...,x_n$ of degrees $d_1,...,d_n$ then 
the radius of convergence of the Hilbert series $h_A(x)$  
is bounded below by the real root of the equation
$\sum_{i=1}^n x^{d_i}=1$ (since for the free algebra 
with such generators $h_A(t)=(1-\sum_{i=1}^n x^{d_i})^{-1}$.}     

\subsection{Modules over $\Bbb C[S_t]$}\label{hcmod}
If $G$ is a finite group then the category of finite dimensional $\Bbb C[G]$-modules
in ${\bold{Rep}}(G)$ is equivalent the category 
${\bold{Rep}}(G\times G)={\bold{Rep}}(G)\boxtimes {\bold{Rep}}(G)$.
Indeed, given $X,Y\in {\bold{Rep}}(G)$, we can take $X\otimes Y\in {\bold{Rep}}(G)$ 
and introduce the action of $\Bbb C[G]$ on it by acting on the first tensor factor, 
which gives the desired equivalence. 
However, the category ${\Bbb C}[S_t]$-fmod of finite dimensional 
${\Bbb C}[S_t]$-modules is not equivalent to 
$\Rep(S_t)\boxtimes \Rep(S_t)$. Indeed, ${\Bbb C}[S_t]$-fmod 
contains a nontrivial invertible object ${\bold s}$ (of order 2), which we may
call the sign representation, and which is $\bold 1$ as an object of 
$\Rep(S_t)$, with the action of $E_{\bold m}$ defined by 
the map $(-1)^{\sum_i im_i}\varepsilon$. 

However, we have the following proposition. 

\begin{proposition} 
If $t$ is transcendental, then the category ${\Bbb C}[S_t]$-fmod 
is equivalent to $\Rep(S_t)\boxtimes \Rep(S_t)\oplus \Rep(S_t)\boxtimes 
\Rep(S_t)\otimes {\bold s}$ as a tensor category. 
\end{proposition} 

\begin{proof}
Let $N(\lambda)$ be a $\Bbb Z_+$-valued function on the set of all partitions 
which takes finitely many nonzero values. Let $X(N)=\oplus_\lambda N(\lambda)X_\lambda$ be the corresponding object of $\Rep(S_t)$. 
Our job is to show that for transcendental $t$, any action of $\Bbb C[S_t]$ on $X(N)$ 
comes from an object of $\Rep(S_t)\boxtimes \Rep(S_t)\oplus \Rep(S_t)\boxtimes 
\Rep(S_t)\otimes {\bold s}$. To this end, let us use the presentation of $\Bbb C[S_t]$ 
as a quotient of the tensor algebra $TE$ by the appropriate relations, introduced above. 
This presentation implies that an action of $\Bbb C[S_t]$ on $X(N)$ 
is determined by a morphism $\rho\in {\rm Hom}_{S_t}(E\otimes X(N),X(N))$ 
satisfying certain quadratic relations. Now observe that the space 
$W:={\rm Hom}_{S_t}(E\otimes X(N),X(N))$ is independent of $t$ 
(it has a basis given by certain diagrams in the partition algebra). 
Moreover, the quadratic relations for $\rho\in W$ depend polynomially on $t$.
These relations define a family of closed subvarieties $Y_t$ in $W$, $t\in \Bbb C$
("representation varieties"), with an action of the group $GL_N=GL_N(\Bbb C):=
\prod_\lambda GL_{N(\lambda)}(\Bbb C)$ by change of basis. 
This family is not necessarily flat; however, since the equations of $Y_t$ are polynomial in $t$ with rational coeffiicients, 
the set of $t$ for which $Y_t$ has a given finite number of $GL_N$-orbits is clearly semialgebraic, defined over $\Bbb Q$.
So if we show that $|Y_t/GL_N|$ is a certain fixed number $m(N)$ 
for sufficiently large integer $t$, it will follow that $|Y_t/GL_N|=m(N)$ 
 for any transcendental $t$. 
 
 To see that $|Y_t/GL_N|=m(N)$ for large integer $t=n$,
 note that, as explained above, $Y_t$ is the variety of representations 
 of $S_n\times S_n$ which restrict to $X_n(N):=\oplus_\lambda N(\lambda)\pi_{\widetilde{\lambda}(n)}$ (whose interpolation is $X(N)$)
  on the diagonal subgroup $(S_n)_{\rm diag}\subset S_n\times S_n$. 
 This variety clearly has finitely many orbits of $GL_N$, which implies the statement. 
 
 It remains to show that $m(N)$ equals the number of objects of $\Rep(S_t)\boxtimes \Rep(S_t)\oplus \Rep(S_t)\boxtimes 
\Rep(S_t)\otimes {\bold s}$ which map to $X(N)$ under the forgetful functor. 
To this end, we will use the following (well known) combinatorial lemma, whose proof is given in the 
appendix at the end of the paper. 

\begin{lemma}\label{comb} For each $C>0$ and $k\in \Bbb Z_+$ there exists $N=N(C,k)\in \Bbb Z_+$ 
such that for each $n\ge N$, if $V=\pi_\mu$ is an irreducible representation of $S_n$ 
such that $\dim V\le C n^k$, then either the first row or 
the first column of $\mu$ has length $\ge n-k$. 
\end{lemma}  

Now let $n$ be large.
Our job is to classify all ways to write $X_n(N)$ as $\oplus_{i=1}^p Y_i\otimes Y_i'$, where $Y_i,Y_i'$ are representations of $S_n$. 
Clearly, $p\le \sum_\lambda N(\lambda)$. Also, since $\dim X_n(N)$ is a polynomial of $n$ of some degree $k$, 
by Lemma \ref{comb}, for large $n$ there are only finitely many possibilities for $Y_i$ and $Y_i'$, and there 
are two options: either both have $\ge n-k$ boxes in the first row, or both have $\ge n-k$ boxes in the first column. This implies the required statement, and completes the proof of the proposition. 
\end{proof} 

\begin{remark} We expect that this proposition holds for all $t\notin \Bbb Z_+$, but 
proving this would require a more refined approach.
\end{remark}

On the other hand, there are many infinite dimensional $\Bbb C[S_t]$-modules 
(i.e. based on an ind-object of $\Rep(S_t)$)
which are not ind-objects of $\Rep(S_t)\boxtimes \Rep(S_t)$. 
Indeed, there are only a countable collection 
of possible eigenvalues of the center of $\Bbb C[S_t]$ 
on ind-objects of $\Rep(S_t)\boxtimes \Rep(S_t)$.  

The category $\Bbb C[S_t]$-mod of (possibly infinite dimensional) 
$\Bbb C[S_t]$-modules may be viewed as 
the category of ``Harish-Chandra bimodules'' 
for $S_t$ (as it is analogous to the category 
of Harish-Chandra bimodules for a semisimple Lie algebra).
It would be interesting to study this category in more detail.  

\section{Schur-Weyl duality for $\Rep(S_t)$}

In this section we discuss a Schur-Weyl duality for $\Rep(S_t)$, 
which generalizes the classical Schur-Weyl duality, and 
is based on the notion of a complex power of a vector space
with a distinguished vector, in the case $t\notin \Bbb Z_+$. 
Under this duality, objects of $\Rep(S_t)$ correspond to 
objects of a parabolic category ${\mathcal O}$ for ${\mathfrak{gl}}_n$. 
This is discussed in much more detail and for general $t$ (including $t\in \Bbb Z_+$) in the
forthcoming paper \cite{EA2}. 

\subsection{Unital vector spaces}

\begin{definition}
A unital vector space is a vector space
$V$ which has a distinguished nonzero vector $\mathbb 1$.
\end{definition} 

It is clear that unital vector spaces form a 
symmetric monoidal category under tensor product. 

Let $(V,\Bbb 1)$ be a unital vector space,
and let $\bar V:=V/\Bbb C\Bbb 1$.
Fix a splitting $\bar V\to V$, and denote its image by $U$;
thus, we have $V=\Bbb C\Bbb 1\oplus U$.
This gives an isomorphism ${\rm Aut}(V,\Bbb 1)\cong GL(U)\ltimes U^*$. 

For a partition $\lambda$, let $\Bbb S^\lambda$ be the
corresponding Schur functor on the category of vector spaces.
Define the induced representation of ${\rm Aut}(V,\Bbb 1)$
given by the formula 
$$
E_\lambda:={\rm Ind}_{GL(U)}^{GL(U)\ltimes U^*}\Bbb S^\lambda U,
$$
i.e., $E_\lambda=SU\otimes \Bbb S^\lambda U$,   
with the obvious action of $GL(U)$, and the action of $U^*$ given by 
differentiation in the first component. Note that 
the action of the Lie algebra 
${\rm Lie}({\rm Aut}(V,\Bbb 1))={\mathfrak{gl}}(U)\ltimes U^*$ 
on $E_\lambda$ naturally extends to an action of the Lie algebra 
${\mathfrak {gl}}(V)$; indeed, we can write 
$E_\lambda$ as 
$$
E_{\lambda}={\rm Hom}^{\rm res}_{{\mathfrak {gl}}(U)\ltimes U}
(U({\mathfrak {gl}}(V)),\Bbb S^\lambda U),
$$ 
where $U$ acts on $\Bbb S^\lambda U$ by zero, and the superscript ''res'' means that we are taking 
the restricted space of homomorphisms with respect to the grading in which $\deg(\Bbb 1)=0$, $\deg(U)=1$
(i.e., the space spanned by homogeneous homomorphisms). Thus, $E_\lambda$ is a 
module for the Harish-Chandra pair $({\mathfrak {gl}}(V),{\rm Aut}(V,\Bbb 1))$. 

\begin{proposition}\label{irre}
(i) The module $E_\lambda$ is 
the contragredient module 
$M(t-|\lambda|,\lambda)^\vee$ 
to the the parabolic Verma module 
$M(t-|\lambda|,\lambda)$ over ${\mathfrak {gl}}(V)$ with 
highest weight $(t-|\lambda|,\lambda)$ (integrating to the subgroup 
${\rm Aut}(V,\Bbb 1)$). 

(ii) If $t\notin \Bbb Z_+$, the module 
$E_\lambda$ is irreducible, and hence is
isomorphic to the Verma module 
$M(t-|\lambda|,\lambda)$. 
\end{proposition}

\begin{proof}
Let $u$ be a highest weight vector of $\Bbb S^\lambda U$ as 
a ${\mathfrak {gl}}(U)$-module (clearly, it has weight $\lambda$). 
Then $u$ is a highest weight vector for 
${\mathfrak {gl}}(V)$ of weight $(t-|\lambda|,\lambda)$,
since ${\rm Id}\circ u=tu$. So we have a natural homomorphism 
$M(t-|\lambda|,\lambda)\to E_\lambda$ sending the highest weight vector 
of $M(t-|\lambda|,\lambda)$ to $u$, which implies (i). 

The characters of the two modules 
are the same, so to prove (ii), it suffices to show that 
for $t\notin \Bbb Z_+$, the module $M(t-|\lambda|,\lambda)$ is irreducible.

Assume that $M(t-|\lambda|,\lambda)$ is reducible. 
Then it must contain a singular vector of weight 
$(t-|\lambda|,\lambda)-m\alpha$, where $\alpha$ is a positive root, and $m$ 
a positive integer (this follows, for instance, from the Jantzen
determinant formula for parabolic Verma modules). 
Then, setting $N=\dim V$ and $\rho=(N,N-1,...,1)$, we must have
$$
(t-|\lambda|,\lambda)+\rho-m\alpha=s_\alpha((t-|\lambda|,\lambda)+\rho).
$$ 
Since the submodule generated by the singular vector 
integrates to ${\rm Aut}(V,\Bbb 1)$, this implies that all the coordinates of the vector 
$$
s_\alpha((t-|\lambda|,\lambda)+\rho)
$$
except the first one form a strictly decreasing sequence. 
Thus, $\alpha=e_1-e_i$ for some $i$. 

Now let $\lambda$ be a partition 
with at most $\dim V-1$ parts and $n\ge
\lambda_1+|\lambda|$. If the Verma module 
$M(t-|\lambda|,\lambda)$ has a submodule with 
highest weight vector of weight
$(t-|\lambda|,\lambda)-m(e_1-e_i)$, 
then we must have 
$$
t-|\lambda|-m+N=\lambda_i+N-i,
$$
i.e., 
$$
t=|\lambda|+\lambda_i+m-i.
$$
Thus, $t$ is an integer. Moreover, 
$\lambda_{i-1}\ge \lambda_i+m\ge m$, so $|\lambda|\ge (i-1)m$, 
so 
$$
t\ge (i-1)m+m-i=i(m-1)\ge 0. 
$$
So $t\in \Bbb Z_+$, as desired. 
\end{proof}

\begin{proposition}\label{inje0}
There is a unique, up to scaling, nonzero 
${\rm Aut}(V,\Bbb 1)$-homomorphism 
$f_{\lambda,n}: S^{\widetilde{\lambda}(n)}V\to E_\lambda$, 
and this homomorphism is injective. 
\end{proposition}

\begin{proof}
By Frobenius reciprocity 
$$
\Hom_{{\rm Aut}(V,\Bbb 1)}
(\Bbb S^{\widetilde{\lambda}(n)}V,E_\lambda)=
\Hom_{GL(U)}(\Bbb S^{\widetilde{\lambda}(n)}(U\oplus \Bbb C),\Bbb S^\lambda U).
$$
According to the branching rules
for general linear groups, this space is $1$-dimensional. 
Thus, there is a unique, up to scaling, nonzero 
${\rm Aut}(V,\Bbb 1)$-homomorphism 
$f_{\lambda,n}: S^{\widetilde{\lambda}(n)}V\to E_\lambda$. 

Let us show that $f_{\lambda,n}$ are injective. 
Assume the contrary, and let $0\ne y\in \Bbb S^{\widetilde{\lambda}(n)}V$ 
be a vector such that $f_{\lambda,n}(y)=0$. 
It is easy to show that by 
applying elements of ${\mathfrak {gl}}
(U)\ltimes U^*$ to $y$, we can map $y$ 
to a nonzero vector $y'\in \Bbb S^\lambda U\subset
\Bbb S^{\widetilde{\lambda}(n)}V$ such that $f_{\lambda,n}(y')=0$. 
But this is a contradiction, since $f_{\lambda,n}|_{\Bbb S^\lambda U}$ is
clearly injective. 
\end{proof}

We will normalize $f_{\lambda,n}$ so that it corresponds to the canonical 
element \linebreak in $\Hom_{GL(U)}(\Bbb S^{\widetilde{\lambda}(n)}(U\oplus \Bbb C),
\Bbb S^\lambda U)$.

Now consider the natural homomorphism
$\phi_{\lambda,n}:\Bbb S^{\widetilde{\lambda}(n)}V\hookrightarrow \Bbb
S^{\widetilde{\lambda}(n+1)}V$, defined by setting the 
second argument of the natural projection 
$$
\Bbb S^{\widetilde{\lambda}(n)}V\otimes V\to \Bbb
S^{\widetilde{\lambda}(n+1)}V
$$
to be $\Bbb 1$. It is easy to show that this homomorphism is nonzero.

\begin{proposition}\label{inje}
The homomorphism $\phi_{\lambda,n}$ is injective. 
\end{proposition}

\begin{proof}
By Proposition \ref{inje0}, the homomorphisms 
$f_{\lambda,n}$ are compatible with the homomorphisms 
$\phi_{\lambda,n}: S^{\widetilde{\lambda}(n)}V\hookrightarrow
S^{\widetilde{\lambda}(n+1)}V$,
i.e. $f_{n+1,\lambda}\circ \phi_{\lambda,n}=f_{\lambda,n}$. 
This implies that $\phi_{\lambda,n}$ 
are injective, as desired. 
\end{proof} 

Now let $\Bbb S^{\lambda,\infty}V$ be the direct limit 
of $\Bbb S^{\widetilde{\lambda}(n)}V$ 
with respect to the inclusions $\phi_{\lambda,n}$:
$\Bbb S^{\lambda,\infty}V=\lim_{n\to \infty}\Bbb S^{\widetilde{\lambda}(n)}V$. 

\begin{proposition}\label{limts} 
Let  $V=\Bbb C\Bbb 1\oplus U$ be a splitting.
Then there is a canonical isomorphism of 
the representation $\Bbb S^{\lambda,\infty}V$ 
with the induced representation 
$E_\lambda$.
\end{proposition}

This proposition shows that the representation $E_\lambda$ in fact does not depend 
on the splitting of $V$, and is functorially attached to $(V,\Bbb 1)$. 

\begin{proof}
By Propositions \ref{inje0},\ref{inje}, we have an embedding 
$f_{\lambda,\infty}=\lim_{n\to \infty}f_{\lambda,n}: \Bbb S^{\lambda,\infty}V\to
SU\otimes \Bbb S^\lambda U$. Comparing the restrictions 
of both modules to $GL(U)$ using branching rules, we see 
that this embedding must be an isomorphism. Indeed, both $GL(U)$-modules 
are multiplicity free, and are direct sums of the irreducible modules 
$\Bbb S^\mu U$, where $\mu$ runs over partitions with at most $\dim V-1$ parts 
such that $\mu_i\ge \lambda_i\ge \mu_{i+1}$ for $i\ge 1$.  
\end{proof}

\subsection{Complex powers of a unital vector space}

Let $x$ be a variable. If $t$ is an arbitrary complex number,
then the function $x^t$ does not have an algebraic meaning. 
On the other hand, the function $(1+x)^t$ does: it is just the formal power series
$$
(1+x)^t=\sum_{m=0}^\infty 
\frac{t(t-1)...(t-n+1)}{n!}x^n.
$$
Similarly, if $V$ is a vector space, then we cannot naturally 
define $V^{\otimes t}$ for arbitrary complex $t$,  
but we can do so if $V$ is a {\it unital} vector space. 

\begin{definition} Let $(V,\Bbb 1)$ be a unital vector space. 
For $t\notin \Bbb Z_+$ define $V^{\otimes t}$ to be the ind-object of the
category $\Rep(S_t)$ given by the formula 
$$
V^{\otimes t}=\oplus_\lambda \Bbb S^{\lambda,\infty}V\otimes
X_\lambda.
$$
\end{definition}

This is clearly an interpolation of 
$V^{\otimes n}=\oplus_\lambda \Bbb S^{\widetilde{\lambda}(n)}V\otimes
\pi_{\widetilde{\lambda}(n)}$ to complex rank. 

Thus we have a complex rank analog of the Schur-Weyl functor, 
$SW_t: \Rep(S_t)^{\rm op}\to {\rm Ind}{\bold{Rep}}({\rm Aut}(V,\Bbb 1))$, 
which is given by the formula 
$$
SW_t(\pi)={\rm Hom}(\pi,V^{\otimes t}), 
$$
and $SW_t(X_\lambda)=E_\lambda$ for all partitions $\lambda$
(note that this is zero if $\lambda$ has more than $\dim V-1$ parts). 

\begin{proposition}\label{hilse}
Assume that $(V,\Bbb 1)$ is a nonnegatively graded (or filtered) 
unital vector space with $V[0]$ spanned by $\Bbb 1$. Let the 
Hilbert series of $V$ be $h_V(x)=1+O(x)$. 
Then the Hilbert series of $V^{\otimes t}$ 
is $h(x)^t$. 
\end{proposition}

\begin{proof}
The proof is by interpolation from integer $n$. 
\end{proof}

Let ${\rm gr}V:=\Bbb C\Bbb 1\oplus \bar V$
be the associated graded of the space $V$ with respect to its natural 
2-step filtration. The object $V^{\otimes t}$ has a natural ascending 
$\Bbb Z_+$-filtration such that ${\rm gr} (V^{\otimes t})=({\rm gr}V)^{\otimes t}$. 
Namely, for each Young diagram $\mu$ we have an ascending filtration on $\Bbb S^\mu V$, 
which is induced by the filtration on $V$; this filtration is compatible with the inclusions 
$\phi_{n,\lambda}: \Bbb S^{\widetilde{\lambda}(n)}V\hookrightarrow \Bbb S ^{\widetilde{\lambda}(n+1)}V$ and thus defines 
a filtration on $\Bbb S^{\lambda,\infty}V$. This gives rise to a filtration on 
$V^{\otimes t}$ by taking direct sum. 

\begin{proposition}\label{grad}
One has 
$$
{\rm gr}V^{\otimes t}\cong S\bar V\otimes (\oplus_\lambda \Bbb S^\lambda \bar V\otimes
X_\lambda).
$$
\end{proposition}

\begin{proof}
This follows immediately from Proposition \ref{limts}.
\end{proof}

\begin{proposition}\label{degree1}
One has 
$$
F_0V^{\otimes t}=\bold 1,\ 
F_1V^{\otimes t}=(V\otimes \h)/\h_0=V\otimes \bold 1\oplus \bar
V\otimes \h_0.
$$ 
\end{proposition}

\begin{proof}
The first statement is obvious. 
To prove the second statement, 
note that $\lambda=\emptyset$ and $\lambda=(1)$
are the only partitions contributing to $F_1$
(this is easily seen by looking at the associated graded object).
Now, the contribution of $\lambda=\emptyset$ is $V$, and 
the contribution of $\h_0=X_{(1)}$ is $\bar V\otimes \h_0$, as desired.   
\end{proof} 

\begin{proposition}\label{monfu} The assignment $V\mapsto V^{\otimes t}$
is a (non-additive) symmetric monoidal functor from the category of unital vector spaces 
to ${\rm Ind}\Rep(S_t)$. 
\end{proposition}

\begin{proof} Let $V,W$ be two unital vector spaces. 
For every nonnegative integer $n$ we have morphisms 
$$
J_n: V^{\otimes n}\otimes W^{\otimes n}\to (V\otimes W)^{\otimes
  n}.
$$
and $J_n'=J_n^{-1}$. 
These morphisms are polynomial in $n$ in an appropriate sense, 
hence they interpolate to morphisms 
$J: V^{\otimes t}\otimes W^{\otimes t}\to (V\otimes W)^{\otimes t}$
and $J': (V\otimes W)^{\otimes t}\to V^{\otimes t}\otimes
W^{\otimes t}$, 
such that $J\circ J'=1$, $J'\circ J=1$. 
Also, $J$ is a symmetric monoidal structure, since 
it is one for integer $t$.  
\end{proof} 

Let ${\rm Res}^{S_t}_{S_{t-1}}: \Rep(S_t)\to \Rep(S_{t-1})$ be the restriction functor. 

\begin{proposition}\label{onefac}
We have a natural isomorphism $\psi_t: V^{\otimes t-1}\otimes V\to {\rm Res}^{S_t}_{S_{t-1}}V^{\otimes t}$,
which commutes with the action of ${\mathfrak{gl}}(V)$. 
\end{proposition}

\begin{proof}
The morphism $\psi_t$ is constructed by interpolation from integer $t$, and it is easy to see that 
it is an isomorphism (this follows from the decomposition of $V\otimes M_{(t-1-|\lambda|,\lambda)}$ 
into irreducible ${\mathfrak{gl}}(V)$-modules). 
\end{proof}

\section{Wreath products} 

In this section we use the notion of a complex power of a unital vector space 
to construct wreath products of complex rank (for any associative algebra $A$). 
In the case when $A$ is a group algebra, this was done in a different way by Knop, \cite{Kn1,Kn2}.  

\subsection{Complex powers of a unital algebra}

Now assume that $V=A$ is a unital associative algebra with unit
$\Bbb 1$. Then, by Proposition \ref{monfu}, $A^{\otimes t}$ is also 
an algebra (in ${\rm Ind}\Rep(S_t)$). This algebra interpolates the algebra 
$A^{\otimes n}$ in ${\rm Ind}{\bold{Rep}}(S_n)$, whose category of modules in ${\rm Ind}{\bold{Rep}}(S_n)$ is 
naturally equivalent to the category of modules over the wreath product 
$\Bbb C[S_n]\ltimes A^{\otimes n}$. Thus we can think of the category of
modules over $A^{\otimes t}$ in ${\rm Ind}\Rep(S_t)$ as an interpolation to 
complex rank of the category of modules over the wreath product
$\Bbb C S_n\ltimes A^{\otimes n}$; it will thus be denoted by 
${\widehat{\Rep}}(S_t\ltimes A^{\otimes t})$ (the hat is used 
to emphasize that the modules are allowed to be infinite dimensional, i.e. 
ind-objects of $\Rep(S_t)$). 

Similarly to the case $A=\Bbb C$ considered above, 
we can define the wreath product algebra $\Bbb C[S_t]\ltimes A^{\otimes t}$ in 
${\rm Ind}\Rep(S_t)$. Namely, the Hopf algebra $\Bbb C[S_t]$ acts 
naturally on $A^{\otimes t}$ (since $A^{\otimes t}$ 
is an algebra in ${\rm Ind}\Rep(S_t)$), so $\Bbb C[S_t]\ltimes A^{\otimes t}$
is defined just as the usual smash product from the theory of Hopf actions, 
formed inside ${\rm Ind}\Rep(S_t)$. Moreover, 
$A^{\otimes t}$-modules is the same thing as 
$\Bbb C[S_t]\ltimes A^{\otimes t}$-modules, in which the
subalgebra $\Bbb C[S_t]$ acts via its canonical action
(see Remark \ref{semidir}). 

Note also that if $A$ is a Hopf algebra, then so are $A^{\otimes t}$
and $\Bbb C[S_t]\ltimes A^{\otimes t}$. Hence, in this case, 
the category of $A^{\otimes t}$-modules is a tensor category
(see Remark \ref{semidir}). \footnote{Note that this is slight abuse of terminology, since 
in this category, only finite dimensional objects (i.e., those 
which are honest objects of $\Rep(S_t)$, rather than ind-objects) are rigid.} 

Let us study the structure of $A^{\otimes t}$ in more detail. 
Define the ordinary algebra $S^t A:=\Hom(\bold 1,A^{\otimes t})$
(in the category of vector spaces). 

\begin{proposition}\label{unenv}
For $t\notin \Bbb Z_+$ the algebra $S^t A:=\Hom(\bold 1,A^{\otimes t})$ is naturally
isomorphic to $U(A)/(\Bbb 1=t)$, where $U(A)$ is the universal
enveloping algebra of $A$ regarded as a Lie algebra. 
\end{proposition}

\begin{proof}
We have a natural linear map $A\to S^t A\hookrightarrow A^{\otimes t}$ 
given as the composition $A\to A\otimes \h\to F_1A^{\otimes t}$, 
where the first map is $1\otimes \iota$, $\iota$ being the 
unit map of the algebra $\h$, and the second map is the map of 
Proposition \ref{degree1}. It is easy to check that this map 
is a Lie algebra homomorphism, such that $\Bbb 1\to t$, 
so it gives rise to a homomorphism 
of associative algebras $\zeta_t: U(A)/(\Bbb 1-t)\to S^t A$. 
This homomorphism interpolates the natural homomorphism 
$\zeta_n: U(A)/(\Bbb 1=n)\to S^nA$ for nonnegative integer $n$, given by 
$$
a\to a_1+...+a_n, a\in A,
$$
where $a_i=1^{\otimes i-1}\otimes a\otimes 1^{n-i}\in A^{\otimes n}$.

To check that $\zeta_t$ is an isomorphism for $t\notin \Bbb Z_+$, 
consider the filtration inn $A$ defined by ${\mathcal F}_0A=\Bbb C\Bbb 1$,
${\mathcal F}_1(A)=A$. Then ${\rm gr}(A)$ is the commutative algebra 
$\Bbb C\Bbb 1\oplus U$, where $U=A/\Bbb C\Bbb 1$, and 
the product of any two elements of $U$ is zero. 
The filtration ${\mathcal F}$ extends naturally to $U(A)$ and to $S^tA$, 
and is preserved by the map $\zeta_t$. Thus, it suffices to check that 
${\rm gr}\zeta_t$ is an isomorphism. We have ${\rm gr}(U(A)/(\Bbb 1-t))\cong 
S{\rm gr}(A)/(\Bbb 1-t)\cong SU$, 
and by the results of the previous section, ${\rm gr}(S^tA)\cong S^t{\rm gr}(A)\cong SU$. 
After these identifications, it is easy to see that ${\rm gr}(\zeta_t)$ becomes the identity map. 
This implies the statement.\footnote{The fact that $\zeta_t$ is an isomorphism for generic $t$ also follows from the fact that 
$\zeta_n$ is surjective for each $n$, and asymptotically injective when $n\to \infty$ 
(i.e., injective in any fixed filtration degree for sufficiently large $n$).  }
\end{proof}

\begin{corollary}
Let $M$ be a left $A^{\otimes t}$-module in 
${\rm Ind}\Rep(S_t)$. Let $\pi\in \Rep(S_t)$. 
Then $\Hom(\pi,M)$ is naturally a representation of the Lie algebra $A$ 
with $\Bbb 1$ acting by multiplicaton by $t$. 
\end{corollary}

\begin{proof}
This follows immediately from Proposition \ref{unenv}, since 
$\Hom(\bold 1,A^{\otimes t})$ acts on $\Hom(\pi,M)$ for any $\pi$. 
\end{proof}

Now we would like to describe the algebra $A^{\otimes t}$ by generators and relations. 
In the classical case, the tensor power algebra $A^{\otimes n}$ can be presented 
as follows. The generators are $A\otimes \Bbb C^n$ (spanned by $a_1,...,a_n,a\in A$), 
and the defining relations are: 
\begin{equation}\label{threerel}
\Bbb 1_i=\Bbb 1_j,\ i\ne j; \
a_ib_i=(ab)_i; \
a_ib_j=b_ja_i, i\ne j,\text{ for }a,b\in A. 
\end{equation}
By analogy, in the complex rank case, as generators we will take $F_1A^{\otimes t}$, which, by Proposition \ref{degree1},
is $A\otimes \h/\h_0$. More precisely, we will use $A\otimes \h$ as generators, 
and include $\h_0=\Bbb C\cdot \Bbb 1\otimes \h_0\subset A\otimes \h$ in the ideal of relations, which incorporates 
the first relation in (\ref{threerel}). 

There are two other relations among the generators in $A\otimes \h$, which interpolate the second 
and the third relation in (\ref{threerel}), respectively. 
To write them, let $ m: \h\otimes \h\to \h$ be the natural commutative product, and $m^*: \h\to \h\otimes \h$ 
be the dual map to $m$. Also let $\mu_A$ be the multiplication in $A$, and $[,]_A$ the commutator in $A$. The first relation is the image of the morphism 
$\xi: A\otimes A\otimes \h\to (A\otimes \h)\oplus (A\otimes \h)^{\otimes 2}$ given by 
$$
\xi=\mu_A\otimes {\rm Id}-{\rm Id}\otimes {\rm Id}\otimes m^*.
$$
The second relation is the image of the morphism 
$\eta: A\otimes \h \otimes A\otimes \h\to (A\otimes \h)\oplus (A\otimes \h)^{\otimes 2}$
given by the formula
$$
\eta={\rm Id}-\sigma_{13}\sigma_{24}-([,]_A\otimes m)\circ \sigma_{23},
$$
where $\sigma_{ij}$ denotes the permutation of the $i$-th and the $j$-th factor. 

\begin{proposition}\label{gerel}
The algebra $A^{\otimes t}$ is 
the quotient of $T(A\otimes \h)$ by the ideal $I$ generated by $\h_0\subset A\otimes \h$, 
${\rm Im}\xi$, and ${\rm Im}\eta$.  
\end{proposition}

\begin{proof}
It is easy to see that we have a natural homomorphism 
$$
\theta_A: T(A\otimes \h)/I\to A^{\otimes t}
$$
 (interpolating the case of integer $t$). 
This homomorphism preserves natural filtrations,
and ${\rm gr}(\theta_A)=\theta_{{\rm gr}A}$, where 
${\rm gr}A=\Bbb C\Bbb 1\oplus \overline{A}$, $\overline{A}:=A/\Bbb C\Bbb 1$, 
with multiplication on $\overline{A}$ being zero. It is easy to check that 
$\theta_{{\rm gr}A}$ is an isomorphism, which implies that so is 
$\theta_A$.   
\end{proof} 

\begin{remark}
In a similar way, one can define the algebra $\Bbb C[S_t]\ltimes A^{\otimes t}$
inside ${\rm IndRep}(S_t)$ by generators and relations. Namely, in the classical setting, 
combining relations (\ref{snrel}) and (\ref{threerel}), we see that 
the algebra $\Bbb C[S_n]\ltimes A^{\otimes n}$ is generated by $a_i$, $a\in A$, and $s_{ij}$, with 
the relations 
\begin{gather*}
s_{ij}^2=1,\ s_{ij}s_{jk}=s_{ik}s_{ij},\ s_{ij}s_{kl}=s_{kl}s_{ij}, \\
\Bbb 1_i=\Bbb 1_j,\ i\ne j; \
a_ib_i=(ab)_i; \
a_ib_j=b_ja_i, i\ne j,\text{ for }a,b\in A,\\
s_{ij}a_i=a_js_{ij},
\end{gather*} 
where different subscripts represent different indices. These relations are easy to interpolate to 
complex rank, similarly to how one does it for (\ref{snrel}) and (\ref{threerel}) separately, and 
one defines the algebra $\Bbb C[S_t]\ltimes A^{\otimes t}$ as the quotient of 
$T(E\oplus A\otimes \h)$ by the interpolation of these relations. 
\end{remark}

\begin{remark}
Proposition \ref{gerel} provides a construction 
of the complex tensor power $V^{\otimes t}$ for $V=\Bbb C\Bbb 1\oplus U$, which does not use 
representation theory of ${\mathfrak{gl}}(V)$. Namely, 
define a unital algebra structure on $V$ by declaring $\Bbb 1$ a unit and 
setting $u_1u_2=0$ for $u_1,u_2\in U$. Then we can define 
$V^{\otimes t}$ to be the algebra defined by the presentation of 
Proposition \ref{gerel}. 
\end{remark}

\begin{remark}
It is easy to see that if $V$ is an $A$-module then $V^{\otimes t}$ is 
naturally an $A^{\otimes t}$-module (even if $\Bbb C\Bbb 1\subset V$ 
is not fixed by $A$). Namely, the action of the generators 
$A\otimes \h$ on $V^{\otimes t}$ is obtained by interpolating from integer $t$. 
Thus for every $A$-module $V$ we get a module $M_t(V,v):=V^{\otimes t}$ (with distinguished vector $v$) over $A^{\otimes t}$ for any choice 
of a nonzero vector $\Bbb 1=v\in V$ (clearly, this vector matters only up to scaling). In particular, by taking invariants, it defines 
a module $\Hom(\bold 1,M_t(V,v))=S^tV$ over the Lie algebra $A$ with $\Bbb 1$ acting by multiplication by $t$. 
For example, as explained in the previous section, if $A={\rm End}V$ and $V$ is finite dimensional, then $\Hom(\bold 1,M_t(V,v))$ is the parabolic Verma module 
over ${\rm End}V={\mathfrak{gl}}(V)$ with highest weight $(t,0,0,...)$ integrable to ${\rm Aut}(V,v)$. Note that this module depends in an essential way 
on the choice of the line in $V$ generated by $v$ (and hence so does $M_t(V,v)$). 
\end{remark}

\subsection{Knop's category $\Rep(S_t\ltimes \Gamma^t)$}

Let $\Gamma$ be a finite group. 
Recall that the irreducible representations of the wreath product
$S_n\ltimes \Gamma^n$ are labeled by 
maps 
$$
\lambda: {\rm Irrep}\Gamma\to {\rm Partitions},
$$
$V\to \lambda_V$, such that $\sum_V |\lambda_V|=n$.  
Namely, the representation corresponding to $\lambda$ is 
$$
\pi_\lambda={\rm Ind}_{\prod_V S_{|\lambda_V|}\ltimes
\Gamma^n}^{S_n\ltimes \Gamma^n}
\bigotimes_V (\pi_{\lambda_V}\otimes V^{\otimes |\lambda_V|}).
$$

By analogy with Deligne's 
construction, F. Knop constructed a Karoubian tensor category 
$\Rep(S_t\ltimes \Gamma^t)$ for complex $t$,
interpolating the categories ${\bold{Rep}}(S_n\ltimes \Gamma^n)$ (\cite{Kn2}). 
More precisely, if $t$ is a nonnegative integer, then the category $\Rep(S_t\ltimes \Gamma^t)$
projects onto the category ${\bold{Rep}}(S_n\ltimes \Gamma^n)$. 

The indecomposable objects $X_\lambda$ of $\Rep(S_t\ltimes \Gamma^t)$ 
are labeled by all maps 
$$
\lambda: {\rm Irrep}\Gamma\to {\rm Partitions},
$$
$V\to \lambda_V$. 
Similarly the case of $\Gamma=1$, the objects $X_\lambda$ interpolate 
the representations $\pi_{\widetilde{\lambda}(n)}$, where 
$\sum_V\widetilde{\lambda}(n)_V=n$, and $\widetilde{\lambda}(n)_V=\lambda_V$ for $V\ne \Bbb C$, while 
$$
\widetilde{\lambda}(n)_{\Bbb C}=(n-|\lambda|,(\lambda_{\Bbb C})_1, (\lambda_{\Bbb C})_2,...).
$$ 

It is easy to check that the tensor category $\Rep(S_t\ltimes \Gamma^t)$ is  a full 
subcategory of the category of modules over the Hopf algebra $\Bbb C[\Gamma]^{\otimes t}$. 

\begin{remark}
For any tensor category ${\mathcal{C}}$, M. Mori in \cite{Mo} defined 
a new tensor category $S_t({\mathcal{C}})$ , such that 
if ${\mathcal C}={\bold{Rep}}(\Gamma)$ then 
$S_t({\mathcal{C}})=\Rep(S_t\ltimes \Gamma^t)$. 
If ${\mathcal{C}}=H$-fmod (finite dimensional modules), where $H$ is a 
Hopf algebra then $S_t({\mathcal{C}})$ is a 
full subcategory of $H^{\otimes t}$-mod. 
\end{remark}

\subsection{The central elements for wreath products} 

Every conjugacy class of the wreath product $S_n\ltimes \Gamma^n$ gives rise to a central element 
in its group algebra (the sum of all elements in the conjugacy class). 
For the class of transpositions, we will denote this
central element just by $\Omega$. For the class defined by a nontrivial 
conjugacy class $C\subset \Gamma$, we'll denote the central 
element by $\Omega_C$. 

Similarly to the case of $S_n$, the elements $\Omega$ and $\Omega_C$
can be interpolated to complex $t$ 
(to define endomorphisms of the identity functor of $\Rep(S_t\ltimes \Gamma^t)$). 
If $t\notin\Bbb Z_+$, this can be done directly by interpolation from large integer $n$. Namely,
it is easy to check that $\Omega$ and $\Omega_C$ 
act on $\pi_{\widetilde{\lambda}(n)}$ by scalars which are 
polynomials in $n$, so we can define them for $t\notin \Bbb
Z_+$ on $X_\lambda$ by substituting $t$ instead of $n$ into
these polynomials. 

\section{Interpolation of degenerate affine Hecke algebras and symplectic reflection algebras}

\subsection{Interpolation of degenerate affine Hecke algebras of
type A}

Let $\h=\h_0\oplus \Bbb C$ 
be the permutation representation of the symmetric group
$S_n, n\ge 1$. Let $\Sigma$ be the set of reflections (i.e.,
transpositions) in $S_n$.

Recall (\cite{Dr}) that the degenerate 
affine Hecke algebra (dAHA) $\Hh_k(n)$ of type $A$ 
(for $k\in \Bbb C$) is the quotient of $\Bbb C[S_n]\ltimes T\h$ by the relations
$$
[y,y']=k\sum_{s,s'\in \Sigma}(sy,s'y')_\h[s,s'], y,y'\in \h.
$$
where $(,)_\h$ is the natural inner product on $\h$. 

This implies that modules over $\Hh_k(n)$ can be 
described categorically as follows. 

\begin{proposition}\label{ahaa}
An $\Hh_k(n)$-module is the same thing as an $S_n$-module $M$ 
equipped with an $S_n$-morphism 
$$
y: \h\otimes M\to M,
$$
such that the morphism 
$$
y\circ (1\otimes y)\circ ((1-\sigma)\otimes 1): \h\otimes \h\otimes M\to M, 
$$
(where $\sigma$ is the permutation 
of components) equals 
$$
k((,)_\h\otimes 1)\circ [\Omega^{13},\Omega^{23}],
$$
where $\Omega_{ij}$ is the Jucys-Murphy morphism $\Omega$ acting in the tensor product of the 
$i$-th and the $j$-th factor. 
\end{proposition} 

The object $\h$ (the permutation representation) 
is defined in the category $\Rep(S_t)$ for any $t$, and has a
natural symmetric pairing $(,)_\h: \h\otimes \h\to \bold 1$. 
Therefore, Proposition \ref{ahaa} allows us to define 
the interpolation of the category of modules over 
$\Hh_k(n)$. Namely, we make 
the following 

\begin{definition}
Let $t\in \Bbb C$. An object of the category $\widehat{\Rep}(\Hh_k(t))$ is an ind-object of Deligne's category 
$\Rep(S_t)$ equipped with 
a morphism 
$$
y: \h\otimes M\to M,
$$
such that the morphism 
$$
y\circ (1\otimes y)\circ ((1-\sigma)\otimes 1): \h\otimes \h\otimes M\to M, 
$$
equals 
$$
k((,)_\h\otimes 1)\circ [\Omega^{13},\Omega^{23}].
$$
Morphisms in $\widehat{\Rep}(\Hh_k(t))$ are morphisms in ${\rm Ind}\Rep(S_t)$
which respect $y$. 
\end{definition}

\begin{remark} Note that similarly to the classical case, 
$k$ can be rescaled without changing the category
(so it can be made $0$ or $1$). 
\end{remark} 

\begin{remark}
Alternatively, we can define the algebra $\Hh_k(t)$ in $\Rep(S_t)$ to be the 
quotient of $\Bbb C[S_t]\ltimes T\h$ by the interpolation of the defining relation 
of $\Hh_k(n)$. Then the category $\widehat{\Rep}(\Hh_k(t))$ may be defined as the category 
of modules over $\Hh_k(t)$ in $\Rep(S_t)$ in which $\Bbb C[S_t]$ acts by the canonical action. 
A similar remark applies to the examples below (degenerate AHA of type B and 
wreath product symplectic reflection algebras, in particular, rational Cherednik algebras).    
\end{remark}

\begin{remark}
The category $\Rep(\Hh_k(t))$ has been studied in more detail in the paper \cite{M}. 
\end{remark} 

\subsection{Interpolation of degenerate affine Hecke algebras of
type B} 

Let $\h$ be the reflection representation of the Weyl group 
$S_n\ltimes \Bbb Z_2^n$ of type $B_n$. Let $\Sigma$ be the set of
reflections in this group conjugate to a transposition in $S_n$,
and $\Sigma_{-1}$ the set of reflections conjugate to
$(-1,1,...,1)$. Recall (see e.g. \cite{RS}) that 
the dAHA $\Hh_{k_1,k_2}(n)$ of type $B_n$ is 
the quotient of $\Bbb C[S_n\ltimes \Bbb Z_2^n]
\ltimes T\h$ by the relations
$$
[y,y']=k_1\sum_{s,s'\in \Sigma}(sy,s'y')_\h[s,s']+
k_2\sum_{s\in \Sigma,s'\in \Sigma_{-1}}(sy,s'y')_\h[s,s'],
$$
where $(,)_\h$ is the natural inner product on $\h$. 

The object $\h$ is defined in the category $\Rep(S_t\ltimes {\Bbb Z}_2^t)$ 
for any $t$, and has a natural symmetric pairing 
$(,)_\h: \h\otimes \h\to \bold 1$. 
Therefore, similarly to the type $A$ case, we can define 
the interpolation of the category of modules over 
$\Hh_{k_1,k_2}(n)$. Namely, we make 
the following 

\begin{definition}
Let $t\in \Bbb C$. An object of the category 
$\widehat{\Rep}(\Hh_{k_1,k_2}(t))$ is an ind-object of Knop's category 
$\Rep(S_t\ltimes {\Bbb Z}_2^t)$ equipped with 
a morphism 
$$
y: \h\otimes M\to M,
$$
such that the morphism 
$$
y\circ (1\otimes y)\circ ((1-\sigma)\otimes 1): \h\otimes \h\otimes M\to M, 
$$
equals 
$$
((,)_\h\otimes 1)\circ
(k_1[\Omega^{13},\Omega^{23}]+k_2[\Omega^{13},
\Omega_{-1}^{23}]).
$$
Morphisms in $\widehat{\Rep}(\Hh_{k_1,k_2}(t))$ are morphisms in
${\rm Ind}\Rep(S_t\ltimes \Bbb Z_2^t)$
which respect $y$. 
\end{definition}

\begin{remark} Note that similarly to the classical case, 
$(k_1,k_2)$ can be rescaled without changing the category. 
\end{remark} 

\subsection{Interpolation of symplectic reflection algebras for wreath products} 

Let $\Gamma\subset SL_2(\Bbb C)$ be a finite subgroup. 
Let $\hbar,k\in \Bbb C$, and fix complex numbers $c_C$, 
for nontrivial conjugacy classes $C\subset \Gamma$.
Also denote by $T_C$ the half-trace of an element $\gamma\in C$ in the
tautological representation: $T_C=\frac{1}{2}\Tr|_{\Bbb
C^2}\gamma$. 

Let $V=(\Bbb C^2)^n$ be the tautological representation 
of the wreath product $S_n\ltimes \Gamma^n$.
This representation has a natural symplectic pairing $\omega(,)$.  
Let $\Sigma$ be the set of elements conjugate 
in $S_n\ltimes \Gamma^n$ to transpositions, and
$\Sigma_C$ the set of elements of $S_n\ltimes \Gamma^n$ conjugate to
elements of the form $(\gamma,1,1,...,1)$, $\gamma\in C$
where $C$ is a nontrivial conjugacy class in $\Gamma$. 

Recall (\cite{EG,EGG}) that the symplectic reflection
algebra $H_{\hbar ,k,c}(\Gamma,n)$ is the quotient of $\Bbb
C[S_n\ltimes \Gamma^n]\ltimes TV$ by the relations 
$$
[y,y']=\hbar\omega(y,y')-k\sum_{s\in \Sigma}\omega(y,(1-s)y')s
-\sum_C \frac{c_C}{1-T_C}\sum_{s\in
\Sigma_C}\omega((1-s)y,(1-s)y')s.
$$

The object $V$ is defined in the category $\Rep(S_t\ltimes \Gamma^t)$ 
for any $t$, and has a natural symplectic pairing 
$\omega: V\otimes V\to \bold 1$. 
Therefore, we can define 
the interpolation of the category of modules over 
$H_{\hbar,k,c}(\Gamma,n)$. Namely, we make 
the following 

\begin{definition}
Let $t\in \Bbb C$. An object of the category 
$\widehat{\Rep}(H_{\hbar,k,c}(\Gamma,t))$ is an ind-object of Knop's category 
$\Rep(S_t\ltimes \Gamma^t)$ equipped with 
a morphism 
$$
y: V\otimes M\to M,
$$
such that the morphism 
$$
y\circ (1\otimes y)\circ ((1-\sigma)\otimes 1): V\otimes V\otimes M\to M, 
$$
equals 
$$
(\omega\otimes 1)\circ 
(\hbar-k(\Omega^3-\Omega^{23})
-\sum_C \frac{c_C}{1-T_C}(\Omega_C^3-\Omega_C^{13}-\Omega_C^{23}+
\Omega_C^{123})).
$$
Morphisms in $\widehat{\Rep}(H_{\hbar,k,c}(\Gamma,n))$ are morphisms in
${\rm Ind}\Rep(S_t\ltimes \Gamma^t)$
which respect $y$. 
\end{definition}
 
\begin{remark} Note that similarly to the classical case, 
$(\hbar,k,c)$ can be rescaled without changing the category. 
\end{remark} 

\subsection{Interpolation of rational Cherednik algebras of type $A$}
 
It is instructive to consider separately the simplest 
special case $\Gamma=1$, i.e., that of the
rational Cherednik algebra of type $A$. In this case, there is no
classes $C$, $V=\h\oplus \h^*$ and the definition is simplified.\footnote{To avoid confusion,
we do not identify $\h$ and $\h^*$ here.} 
 Namely, we have

\begin{definition}\label{catdes1}
An object of the category 
$\widehat{\Rep}(H_{\hbar,c}(t))$ is an ind-object of Deligne's category 
$\Rep(S_t)$ equipped with two morphisms 
$$
x: \h^*\otimes M\to M,\ y: \h\otimes M\to M,
$$
satisfying the following conditions: 

(i) The morphism 
$$
x\circ (1\otimes x)\circ ((1-\sigma)\otimes 1): 
\h^*\otimes \h^*\otimes M\to M
$$ 
is zero;  

(ii) The morphism 
$$
y\circ (1\otimes y)\circ ((1-\sigma)\otimes 1): 
\h\otimes \h\otimes M\to M
$$
is zero;  

(iii)
The morphism 
$$
y\circ (1\otimes x)-x\circ (1\otimes y)\circ (\sigma\otimes 1): \h\otimes
\h^*\otimes M\to M, 
$$
when regarded as an endomorphism of $\h^*\otimes M$, 
equals 
$$
\hbar-c(\Omega^2-\Omega^{12}).
$$
Morphisms in this category are morphisms of the underlying ind-objects of $\Rep(S_t)$ respecting 
$x,y$.  
\end{definition}

\begin{remark}
The category $\widehat{\Rep}(H_{1,c}(t))$ is studied in much more detail in the paper \cite{EA1}. 
\end{remark}

\section{Appendix: Proof of Lemma \ref{comb}}

Let $\mu$ be a Young diagram with $|\mu|=n$, 
such that $\dim \pi_\mu\le C n^k$, and 
$d$ be the length of its first row. For $1\le i\le d$, let 
$c_i$ be the length of the $d-i+1$-th column of $\mu$. 
Let $\mu'$ be obtained from $\mu$ 
by deleting the first row. Then by the hooklength formula,  
$$
\dim \pi_{\mu}=\dim \pi_{\mu'}\binom{n}{d}\prod_{i=1}^{d}\left(1+\frac{c_i-1}{i}\right)^{-1}.
$$
Since $\sum c_i=n$, by the arithmetic and geometric mean inequality
(as in the proof of Claim 1 in Section 3.2.1 of \cite{EFP}), we have 
 $$
 \prod_{i=1}^{d}\left(1+\frac{c_i-1}{i}\right)\le 
 \prod_{i=1}^{d}c_i\le \left(\frac{\sum_{i=1}^{d}c_i}{d}\right)^{d}=
 \left(\frac{n}{d}\right)^{d}, 
 $$
 so  we get 
 $$
\dim \pi_{\mu}\ge \binom{n}{d}\left(\frac{n}{d}\right)^{-d}.
$$ 
The same bound holds if $d$ is the length of the first column. 
Hence, it holds if $d$ is the maximum of the lengths of the first 
row and the first column of $\mu$, which we assume from now on. 

Our job is to show that for large enough $n$, one must have $n-d\le K$ for 
some fixed $K$. Then it will follow from the hooklength formula that we actually have $d\ge n-k$ 
for large $n$. 

Since $\dim \pi_\mu\le Cn^k$, taking logs, we get the inequality
$$
\log\binom{n}{d}-d\log\left(\frac{n}{d}\right)\le k\log n+\log C.
$$
Let us divide this inequality by $n$, and use Stirling's formula. 
Setting $x=1-d/n$, 
after a short calculation we obtain
\begin{equation}\label{basineq}
x\log\frac{1}{x}=O\left(\frac{(k+1/2)\log n}{n}\right). 
\end{equation}

But we know that $d\ge \sqrt{n}$ (as $\mu$ fits in the $d$ by $d$ square). 
Hence, $x\le 1-n^{-1/2}$. Formula (\ref{basineq}) therefore implies that for large $n$, $x$ must actually 
be close to $0$. But then we must have $x=O(1/n)$, i.e. $n-d=xn$ is bounded, as desired.

\end{document}